\documentclass[11pt]{article}

\usepackage{amsmath,amssymb,amsbsy,amsfonts,amsthm,latexsym,
            amsopn,amstext,amsxtra,euscript,amscd,color,breqn,mathrsfs}

\PassOptionsToPackage{hyphens}{url}\usepackage{hyperref}
    
\usepackage[capbesideposition=outside,capbesidesep=quad]{floatrow}

\restylefloat{table}
            
\usepackage{multirow,caption}

\newfont{\teneufm}{eufm10}
\newfont{\seveneufm}{eufm7}
\newfont{\fiveeufm}{eufm5}

\usepackage{systeme}

\usepackage{xcolor}
\usepackage{tikz}
\usetikzlibrary{shapes}
\usetikzlibrary{decorations.markings}

\usepackage{pstricks}
\usetikzlibrary{shapes,arrows}
 \usepackage{pstricks-add}
\usepackage{epsfig}
 \usepackage[space]{grffile} 
 \usepackage{etoolbox} 
 \makeatletter 
 \patchcmd\Gread@eps{\@inputcheck#1 }{\@inputcheck"#1"\relax}{}{}
 \makeatother

\newtheorem{thm}{Theorem}

\newtheorem{rem}[thm]{Remark}

\newcommand{\Tr}{{\rm Tr}}
\newcommand{\Trn}{{\rm Tr}_n}

\newcommand{\cB}{\mathscr{B}}

\def\+{\oplus}

\def\cA{{\mathcal A}}
\def\cB{{\mathcal B}}
\def\cC{{\mathcal C}}

\def\cN{{\mathcal N}}

\def\F{{\mathbb F}}

\def\Fn{{\mathbb{F}_{p^n}}}

\def\00{{\bf 0}}
\def\11{{\bf 1}}
\def\+{\oplus}

\def\\{\cr}
\def\({\left(}
\def\){\right)}

\newcommand{\cardinality}[1]{\# #1}

\usepackage[colorinlistoftodos,prependcaption,textsize=tiny]{todonotes}

\providecommand{\newoperator}[3]{%
  \newcommand*{#1}{\mathop{#2}#3}}

\newoperator{\FD}{\mathrm{FD}}{\nolimits}

\newcommand{\sA}{\mathscr{A}}

\newcommand{\sX}{\mathscr{X}}
\newcommand{\sY}{\mathscr{Y}}

\begin{document}
\title{\bf Using double Weil sums in finding the $c$-Boomerang Connectivity Table for monomial functions on finite fields}
\author{Pantelimon~St\u anic\u a  \\ 
Applied Mathematics Department, \\
Naval Postgraduate School, Monterey, USA. \\
E-mail: pstanica@nps.edu}

 \maketitle 

\begin{abstract}
In this paper we characterize the  $c$-Boomerang Connectivity Table (BCT),  $c\neq 0$ (thus, including the classical $c=1$ case), for all monomial function $x^d$ in terms of characters and Weil sums on the finite field~$\F_{p^n}$, for an odd prime~$p$. We further simplify these expressions for the Gold functions $x^{p^k+1}$ for all $1\leq k<n$, and $p$ odd. It is the first such attempt for a complete description  for the classical BCT and its relative $c$-BCT, for all parameters involved.
\end{abstract}

{\bf Keywords:} 
Finite fields,
characters,
$p$-ary functions, 
$c$-differentials,  
differential uniformity,
boomerang uniformity,
double Weil sums
\newline
{\bf MSC 2000}: 11L07, 11L40, 11T06, 11T24, 94A60.

\section{Introduction and basic definitions}

For the first time, in this paper we find a characterization of the boomerang connectivity table and its uncle, $c$-boomerang connectivity table ($c$-BCT)~\cite{S20}, for all monomial functions in terms of characters of the relevant finite field (all characteristics). We further detail that characterization for the Gold functions $x^{p^k+1}$ for all $1\leq k<n$, where $p$ is an odd prime.  Since our method mostly relies on finding some double Weil sums, it may have an interest beyond its applicability in the computation of the $c$-BCT.
   
For a positive integer $n$ and $p$ a prime number, we let $\F_{p^n}$ be the  finite field with $p^n$ elements, and $\F_{p^n}^*=\F_{p^n}\setminus\{0\}$ be the multiplicative group (for $a\neq 0$, we often write $\frac{1}{a}$ to mean the multiplicative inverse of $a$).  We use $|{S}|$ to denote the cardinality of a set $S$ and $\bar z$, for the complex conjugate.
We call a function from $\F_{p^n}$   to $\F_p$  a {\em $p$-ary  function} on $n$ variables. For positive integers $n$ and $m$, any map $F:\F_{p^n}\to\F_{p^m}$   is called a {\em vectorial $p$-ary  function}, or {\em $(n,m)$-function}. When $m=n$, $F$ can be uniquely represented as a univariate polynomial over $\F_{p^n}$ (using some identification, via a basis, of the finite field with the vector space) of the form
$
F(x)=\sum_{i=0}^{p^n-1} a_i x^i,\ a_i\in\F_{p^n}.
$
For $f:\F_{p^n}\to \F_p$ we define the absolute trace $\Trn:\F_{p^n}\to \F_p$, given by $\displaystyle \Trn(x)=\sum_{i=0}^{n-1} x^{p^i}$ (we will denote it by $\Tr$, if the dimension is clear from the context).  The reader can consult~\cite{LN97} for more on this and related notions in finite fields.
  
Given a vectorial $p$-ary  function $F$, the derivative of $f$ with respect to~$a \in \F_{p^n}$ is the $p$-ary  function
$
 D_{a}F(x) =  F(x + a)- F(x), \mbox{ for  all }  x \in \F_{p^n}.
$ 
For an $(n,n)$-function $F$, and $a,b\in\F_{p^n}$, we let $\Delta_F(a,b)=\cardinality{\{x\in\F_{p^n} : D_{a}f(x)=b\}}$. We call the quantity
$\delta_F=\max\{\Delta_F(a,b)\,:\, a,b\in \F_{p^n}, a\neq 0 \}$ the {\em differential uniformity} of $F$. If $\delta_F= \delta$, then we say that $F$ is differentially $\delta$-uniform. 

For the interested reader, we point to~\cite{Bud14,CH1,CH2,CS17,MesnagerBook,Tok15} for a proper background on  Boolean and $p$-ary functions.

As a follow up to Wagner's work~\cite{Wag99} on the boomerang attack against block ciphers (see also~\cite{BK99,KKS00,BDK02,Kim12}) and  Cid et al.~\cite{Cid18} who introduced the theoretical tool called  the Boomerang Connectivity Table (BCT) and Boomerang Uniformity, we defined in~\cite{S20} the $c$-Boomerang Connectivity Table ($c$-BCT) and analyzed some known perfect nonlinear as well as the inverse function in even and odd characteristics.  

Let $F$ be a permutation on $\F_{p^n}$  and $(a,b)\in\F_{p^n}\times \F_{p^n}$. We define the entries of the {\em Boomerang Connectivity Table} \textup{(}{\em BCT}\textup{)} by
\[
\cB_F(a,b)=\cardinality \{x\in\F_{2^n}|F^{-1} (F(x)+b)-F^{-1}(F(x+a)+b)=a  \},
\]
where $F^{-1}$ is the compositional inverse of $F$, and the {\em boomerang uniformity} of $F$  is
$\displaystyle
\beta_F=\max_{a,b\in\Fn*} \cB_F(a,b).
$
We also say that $F$ is a $\beta_F$-uniform BCT function. Surely, $\Delta_F(a,b)=0,2^n$ and $\cB_F(a,b)=p^n$  whenever $ab=0$. We know that $\delta_F=\delta_{F^{-1}}$, $\beta_F=\beta_{F^{-1}}$, and  for permutations, $\beta_F\geq \delta_F$ and they are equal for APN permutations.    
 We mention here that this concept became an object of study for many recent papers~\cite{BC18, BPT19, CV19,Li19, LiHu20, Mes19,TX20}, to mention just a few.  
 
  Li et al.~\cite{Li19} (see also~\cite{Mes19}) observed that
\begin{equation*}
\begin{split}
\label{eq:boom-diff}
\cB_F(a,b)&=\cardinality \left\{ (x,y)\in\Fn\times \Fn\,\Large\big|\, \substack{F(x)+F(y)=b \\  F(x+a)+F(y+a)=b } \right\}, \\
& =\sum_{\gamma\in\Fn} \cardinality  \left\{ x \in\Fn \,\big|\, D_\gamma F(x)=b \text{ and }  D_\gamma F(x+a)=b   \right\},
\end{split}
\end{equation*}
and therefore, the concept can be extended to non-permutations, since it avoids the inverse of~$F$.

Based upon our prior $c$-differential concept~\cite{EFRST20} (see also~\cite{HMRS20,RS20,SG20,YZ20} for very recent work on that topic), we extended this notion recently in~\cite{S20} to the $c$-Boomerang Connectivity Table.
For a $p$-ary $(n,m)$-function   $F:\F_{p^n}\to \F_{p^m}$, and $c\in\F_{p^m}$, the ({\em multiplicative}) {\em $c$-derivative} of $F$ with respect to~$a \in \F_{p^n}$ is the  function
\[
 _cD_{a}F(x) =  F(x + a)- cF(x), \mbox{ for  all }  x \in \F_{p^n}.
\]

For an $(n,n)$-function $F$, and $a,b\in\F_{p^n}$, we let the entries of the $c$-Difference Distribution Table ($c$-DDT) be defined by ${_c\Delta}_F(a,b)=\cardinality{\{x\in\F_{p^n} : F(x+a)-cF(x)=b\}}$. We call the quantity
\[
\delta_{F,c}=\max\left\{{_c\Delta}_F(a,b)\,|\, a,b\in \F_{p^n}, \text{ and } a\neq 0 \text{ if $c=1$} \right\}\]
the {\em $c$-differential uniformity} of~$F$. If $\delta_{F,c}=\delta$, then we say that $F$ is differentially $(c,\delta)$-uniform (or that $F$ has $c$-uniformity $\delta$, or for short, {\em $F$ is $\delta$-uniform $c$-DDT}). We can recover all the classical perfect and almost perfect nonlinear functions, taking $c=1$.
It is easy to see that if $F$ is an $(n,n)$-function, that is, $F:\F_{p^n}\to\F_{p^n}$, then $F$ is PcN if and only if $_cD_a F$ is a permutation polynomial.

Further, 
for an $(n,n)$-function $F$, $c\neq 0$, and $(a,b)\in\Fn\times \Fn$,  we define~\cite{S20} the {\em $c$-Boomerang Connectivity Table}  \textup{(}$c$-BCT\textup{)} entry at $(a,b)$ to be
{\small
\begin{equation*}
\label{eq:originalBCT}
  _c\cB_F(a,b)=\cardinality \left\{ x\in\Fn\,\Big|\,  F^{-1}(c^{-1} F(x+a)+b) -F^{-1}(cF(x)+b)=a \right\}.
\end{equation*}
}
and  the {\em $c$-boomerang uniformity} of $F$ is 
$\displaystyle
\beta_{F,c}=\max_{a,b\in\Fn*} {_c}\cB_F(a,b).
$
If $\beta_{F,c}=\beta$, we also say that $F$ is a $\beta$-uniform $c$-BCT function.
We showed in~\cite{S20} that we can avoid inverses, thus   allowing the definition to be extended to all $(n,m)$-function, not only permutations. Precisely,   the entries of the $c$-Boomerang Connectivity Table at $(a,b)\in\F_{p^n}\times\F_{p^n}$  can be given by
\begin{align*}
_c\cB_F(a,b)&=\cardinality \left\{ (x,y)\in\Fn\times \Fn\,\Big|\, \Large\substack{F(y)-cF(x) =b \\  F(y+a)- c^{-1} F(x+a)=b }  \right\}\\
& =\sum_{\gamma\in\Fn} \cardinality  \left\{ x \in\Fn \,\big|\, \Large \substack{_cD_\gamma F(x)=b \text{ and }  _{c^{-1}}D_\gamma F(x+a)=b \\
\text{the $c$-boomerang system}}  \right\}.
\end{align*}

  
  The exact computation of the differential and/or boomerang uniformity and its relative with respect to $c$ seems to be quite difficult, even for monomials.
  It is the purpose of this paper to characterize the $c$-BCT ($c\neq 0$) for all monomials $x^d$ in terms of characters on the finite field~$\F_{p^n}$, where $p$ is any prime number. We use that characterization to further describe the $c$-BCT for all Gold functions $x^{p^k+1}$, $1\leq k<n$, $p$ odd, and $c\neq 0$.  In particular, our result can be seen as a significant generalization of the known results, where $p=2$, $\gcd(n,k)=1,2$, in which case the boomerang uniformity is 2, respectively, 4.

\section{A description of the $c$-BCT of the power map $x^d$ in terms of characters on $\F_{p^n}$} 
\label{sec3}

We concentrate here on the $c$-boomerang uniformity of the power maps $x^d$
over finite field $\mathbb{F}_{p^n}$. Let $G$ be the Gauss' sum $\displaystyle G(\psi,\chi)=\sum_{z\in\F_q^*} \psi(z)\chi(z)$, where $\chi,\psi$, are additive, respectively, multiplicative characters of $\F_q$, $q=p^n$. Below, we let $\chi_1(a)=\exp\left(\frac{2\pi \imath \Tr(a)}{p}\right)$ be the principal additive character, and $\psi_k\left(g^\ell\right)=\exp\left(\frac{2\pi ik\ell}{q-1}\right)$ be the $k$-th multiplicative character of $\F_q$, $0\leq k\leq q-2$. We let $\psi_1$ be the generator of the cyclic group of multiplicative characters.
\begin{thm}
\label{thm:cBU_Char}
Let $F(x)=x^d$ be a monomial function $\F_q$, $q=p^n$, $p$ a prime number. Let $c\in\F_q^*$ and $b\in\F_q$. Then, the $c$-Boomerang Connectivity Table entry  $_c\cB_F(a,ab)$ at $(a,ab)$, $a\neq 0$, is given by 
{\small
\[
\frac{1}{q}\left({_c}\Delta_{F}(1,b)+{_{c^{-1}}}\Delta_{F}(1,b)\right)+1 +\frac1{q^2}\sum_{\alpha,\beta\in\F_q,\alpha\beta\neq 0} \chi_1(-b(\alpha+\beta))\, S_{\alpha,\beta}\, S_{-\alpha c,-\beta c^{-1}},
\]
}
with
\begin{align*}
S_{\alpha,\beta}&=\sum_{x\in \F_q} \chi_1\left(\alpha x^d\right)\chi_1\left(\beta(x+1)^d\right)\\
&= \frac{1}{(q-1)^2} \sum_{j,k=0}^{q-2}   G(\bar\psi_j,\chi_1) G(\bar\psi_k,\chi_1) \sum_{x\in \F_q} \psi_1\left((\alpha x^d)^j (\beta(x+1)^d)^k\right).
\end{align*}
\end{thm}
\begin{proof}
For $b \neq 0$ and fixed $c\neq 1$, the $c$-boomerang uniformity of $x^d$ is given by $\displaystyle \max_{b \in \mathbb{F}_{p^n}^*}~  _cB_F(1,b) $, where $_cB_F(1,b)$ is the number of solutions in $\mathbb{F}_{q}\times \mathbb{F}_{q}$, $q=p^n$, of the following system
\begin{equation}
\label{eq:eq3.1}
\begin{cases}
 x^d-cy^d=b \\
(x+1)^d-c^{-1} (y+1)^d=b.
\end{cases}
\end{equation}
We know (and easy to argue)  that the number  $N(b)$  of solutions $(x_1,\ldots,x_n)\in\F_q^n$, for $b$ fixed, of an equation $f(x_1,\ldots,x_n)=b$ is
\begin{equation}
\begin{split}
\label{eq:char_eq}
\cN(b)
&= \frac{1}{q}\sum_{x_1,\ldots,x_n\in \F_q}\sum_{\alpha\in\F_q} \chi_1\left(\alpha \left( f(x_1,\ldots,x_n)- b\right)\right)\\
&=\frac{1}{q}\sum_{x_1,\ldots,x_n\in \F_q}\sum_{\chi\in \widehat{\F_q}}\chi(f(x_1,\ldots,x_n))\overline{\chi(b)},
\end{split}
\end{equation}
where $\widehat{\F_q}$ is the set of all additive characters of $\F_q$, and $\chi_1$ is the principal additive character of  $\F_q$. Next, note  that the number of solutions  $(x_1,\ldots,x_n)\in\F_q^n$ of a system $f_1(x_1,\ldots,x_n)=b_1$,  $f_2(x_1,\ldots,x_n)=b_2$ is exactly
\[
\frac{1}{q^2}\sum_{x_1,\ldots,x_n\in \F_q}\sum_{\alpha,\beta\in\F_q} \chi_1\left(\alpha\left(f_1(x_1,\ldots,x_n)-b_1 \right) \right) \chi_1\left(\beta\left(f_2(x_1,\ldots,x_n)-b_2 \right) \right).
\]
For our system~\eqref{eq:eq3.1}, we see that the number of solutions for some $a,b$ fixed is therefore
\allowdisplaybreaks
\begin{align*}
&\cN_{b;c}
=\frac{1}{q^2}\sum_{x,y\in \F_q}\sum_{\alpha,\beta\in\F_q} \chi_1\left(\alpha\left(x^d-cy^d-b \right) \right) \chi_1\left(\beta\left((x+1)^d-c^{-1} (y+1)^d-b \right) \right)\\
&=\frac{1}{q^2}\sum_{x,y\in \F_q}\sum_{\alpha,\beta\in\F_q} \chi_1(-b(\alpha+\beta)) \chi_1\left(\alpha x^d+\beta(x+1)^d\right) \overline{ \chi_1\left(\alpha cy^d+ \beta c^{-1} (y+1)^d  \right)}\\
&= \frac{1}{q^2}\sum_{\alpha,\beta\in\F_q} \chi_1(-b(\alpha+\beta))
\sum_{x\in \F_q}
\chi_1\left(\alpha x^d+\beta(x+1)^d\right) \sum_{y\in \F_q} \overline{ \chi_1\left(\alpha cy^d+ \beta c^{-1} (y+1)^d  \right)}.
\end{align*}
We now, rewrite the above expression as (the term $q^2$ comes from $\alpha=\beta=0$)
\allowdisplaybreaks
\begin{align*}
&q^2\cN_{b;c}-q^2= \sum_{\alpha\in\F_q,\beta=0} \chi_1(-b \alpha)\sum_{x\in \F_q}
\chi_1\left(\alpha x^d\right) \sum_{y\in \F_q}   \chi_1\left(-\alpha cy^d \right)\\
&\qquad +\sum_{\alpha=0,\beta\in\F_q} \chi_1(-b \beta)
\sum_{x\in \F_q}
\chi_1\left(\beta(x+1)^d\right) \sum_{y\in \F_q}  \chi_1\left(-\beta c^{-1} (y+1)^d  \right)\\
&\qquad +\sum_{\alpha,\beta\in\F_q,\alpha\beta\neq 0} \chi_1(-b(\alpha+\beta))
\sum_{x\in \F_q}
\chi_1\left(\alpha x^d+\beta(x+1)^d\right) \\
&\qquad\qquad\qquad\qquad\qquad\qquad\cdot  \sum_{y\in \F_q} \overline{ \chi_1\left(\alpha cy^d+ \beta c^{-1} (y+1)^d  \right)}\\
&=\sum_{x,y\in \F_q} \sum_{\alpha\in\F_q}\chi_1\left(\alpha\left( x^d-cy^d-b  \right) \right)\\
&\qquad +\sum_{x,y\in \F_q} \sum_{\beta\in\F_q}\chi_1\left(\beta\left( (x+1)^d-c^{-1}(y+1)^d-b  \right) \right)\\ 
&\qquad +\sum_{\alpha,\beta\in\F_q,\alpha\beta\neq 0} \chi_1(-b(\alpha+\beta))
\sum_{x\in \F_q}
\chi_1\left(\alpha x^d+\beta(x+1)^d\right)\\
&\qquad\qquad\qquad\qquad\qquad\qquad\cdot  \sum_{y\in \F_q} \overline{ \chi_1\left(\alpha cy^d+ \beta c^{-1} (y+1)^d  \right)}\\
&=\sum_{x,y\in \F_q} \sum_{\alpha\in\F_q}\chi_1\left(\alpha\left( x^d-cy^d-b  \right) \right)+\sum_{x,y\in \F_q} \sum_{\beta\in\F_q}\chi_1\left(\beta\left( x^d-c^{-1}y^d-b  \right) \right)\\ 
&\qquad +\sum_{\alpha,\beta\in\F_q,\alpha\beta\neq 0} \chi_1(-b(\alpha+\beta))
\sum_{x\in \F_q}
\chi_1\left(\alpha x^d+\beta(x+1)^d\right)\\
&\qquad\qquad\qquad\qquad\qquad\qquad\cdot  \sum_{y\in \F_q} \overline{ \chi_1\left(\alpha cy^d+ \beta c^{-1} (y+1)^d  \right)}\\
&=q\left( {_c}\Delta_{F}(1,b)+ {_{c^{-1}}}\Delta_{F}(1,b) \right)+\sum_{\alpha,\beta\in\F_q,\alpha\beta\neq 0} \chi_1(-b(\alpha+\beta))\, S_{\alpha,\beta}\, S_{-\alpha c,-\beta c^{-1}},
\end{align*}
where $S_{\alpha,\beta} =\sum_{x\in \F_q} \chi_1\left(\alpha x^d+\beta(x+1)^d\right)$ (the last identity, involving the $c$-DDT entries, follows from Equation~\eqref{eq:char_eq}).

Further,~\cite[Equation (5.17)]{LN97} relates the additive character $\chi$ to the cyclic group (of cardinality $q-1$) of all multiplicative characters $\psi$ of $\F_q$, via
\begin{align*}
\chi_1(w)&=\frac{1}{q-1} \sum_{z\in\F_q^*} \chi_1(z)\sum_{j=0}^{q-2} \psi_j(w) \overline{\psi_j(z)}
=\frac{1}{q-1} \sum_{j=0}^{q-2} G(\bar\psi_j,\chi_1) \psi_j(w),
\end{align*}
where $G$ is the Gauss' sum $\displaystyle G(\psi,\chi)=\sum_{z\in\F_q^*} \psi(z)\chi(z)$.
Using this, we get
\allowdisplaybreaks
\begin{align*}
S_{\alpha,\beta}&=\sum_{x\in \F_q} \chi_1\left(\alpha x^d\right)\chi_1\left(\beta(x+1)^d\right)\\
&=\frac{1}{(q-1)^2}  \sum_{x\in \F_q} \sum_{j,k=0}^{q-2} G(\bar\psi_j,\chi_1) G(\bar\psi_k,\chi_1)\psi_j\left(\alpha x^d\right)  \psi_k\left(\beta(x+1)^d\right),
\end{align*}
from which we infer our last identity.
\end{proof}

\begin{rem}
In the previous theorem and the next ones, we could have embedded the differential entries ${_c}\Delta_{F}(1,b),{_{c^{-1}}}\Delta_{F}(1,b)$ into the character sums, but we wanted to point out how the $c$-BCT entries depend upon the $c$-DDT entries.
\end{rem}
 \begin{rem}
 Surely, we could have written the previous theorem for any function $F$, but we simply wanted it for the Gold functions from the next section. We may come back to that idea for other functions. Also, the case $p=2$ will be treated in a separate paper as the Weil sums results are rather different.
 \end{rem}

\section{The $c$-BCT for all Gold functions $x^{p^k+1}$, $p$ odd}
\label{sec4}
We will now use this approach to push  even further the above result for the Gold function. It is perhaps the first result of this type that computes the boomerang uniformity and its relative, the $c$-boomerang uniformity for all functions in this class (we gave a lower bound in~\cite{S20}
the $c$-boomerang uniformity). 
We will, in fact, find all entries in the $c$-BCT, including $c=1$, as well,  for all $c\neq 0$.

We shall make use of the following results from Coulter~\cite{Co98_1,Co98} (we make slight changes in notations and combine various results), who generalized a result of Carlitz~\cite{Car80}. Let $1\leq k<n$, $e=\gcd(n,k)$, and $\mathscr{S}_k(A,B)=\sum_{x\in \F_q} \chi_1\left(A x^{p^k+1}+ B x\right)$. We let $\eta=\psi_{(q-1)/2}$ be the quadratic character of $\F_q$.
\begin{thm}[\textup{\cite{Co98_1}}]
\label{thm:Co98_1}
Let $q=p^n$, $n=2m\geq 2$, $p$ and odd prime, $1\leq k<n$, $e=\gcd(n,k)$. Then:
\begin{enumerate}
\item[$(1)$] If $n/d$ is odd, then 
\[
\mathscr{S}_k(A,0)=
\begin{cases}
(-1)^{n-1}p^m \eta(A) &\text{ if } p\equiv 1\pmod 4\\
(-1)^{n-1}p^m i^n \eta(A) &\text{ if } p\equiv 3\pmod 4.
\end{cases}
\]
\item[$(2)$] If $n/d$ is even, then
\[
\mathscr{S}_k(A,0)=
\begin{cases}
 (-1)^{\frac{m}{e}}\, p^m\,  &\text{ if } A^{\frac{q-1}{p^e+1}}\neq (-1)^{\frac{m}{e}}\\ 
 (-1)^{\frac{m}{e}+1}\, p^{m+e}\,&\text{ if } A^{\frac{q-1}{p^e+1}}= (-1)^{\frac{m}{e}}. 
\end{cases}
\]
\end{enumerate}
\end{thm}

\begin{thm}[\textup{\cite{Co98}}]
\label{thm:Co98}
Let $q=p^n$, $n\geq 2$, $p$ an  odd prime, $1\leq k<n$, $e=\gcd(n,k)$. Let $f(x)=A^{p^k} x^{p^{2k}}+Ax$, for some nonzero $A$. The following statements hold:
\begin{enumerate}
\item[$(1)$]
If $f$ is a permutation polynomial over $\F_q$, and  $x_0$ is the unique element such that $f(x_0)=-B^{p^k},B\neq 0$, then:
\begin{itemize}
\item[$(i)$] If $\frac{n}{e}$ is odd, then 
\[
\mathscr{S}_k(A,B)=
\begin{cases}
(-1)^{n-1} \sqrt{q}\,\eta(-A)\,\overline{\chi_1(Ax_0^{p^k+1})} &\text{ if } p\equiv 1\pmod 4\\
(-1)^{n-1} \imath^{3n} \sqrt{q}\,\eta(-A)\,\overline{\chi_1(Ax_0^{p^k+1})} &\text{ if } p\equiv 3\pmod 4.
\end{cases}
\]
\textup{(}where the solution $x_0=-\frac{1}{2}\sum_{j=0}^{\frac{n}{e}-1} (-1)^j A^{-\frac{p^{(2j+1)k}+1}{p^k+1}} B^{p^{(2j+1)k}}$\textup{)}.
\item[$(ii)$] If $\frac{n}{e}$ is even, then $n=2m$, $A^{\frac{q-1}{p^e+1}}\neq (-1)^{\frac{m}{e}}$ and
\[
\mathscr{S}_k(A,B)=(-1)^{\frac{m}{e}} p^m\, \overline{\chi_1(Ax_0^{p^k+1})}.
\]
\end{itemize}
\item[$(2)$] If $f$ is not a permutation polynomial, then, for $B\neq 0$, $\mathscr{S}_k(A,B)=0$, unless, $f(x)=-B^{p^k}$ has a solution $x_0$ \textup{(}this can only happen if $\frac{n}{e}$ is even with $n = 2m$, and $A^{\frac{q-1}{p^e+1}}= (-1)^{\frac{m}{e}}$\textup{)}, in which case
\[
\mathscr{S}_k(A,B)=(-1)^{\frac{m}{e}+1}p^{m+e} \overline{\chi_1(Ax_0^{p^k+1})}.
\]
\end{enumerate}
\end{thm}

The proof of our results are long and complicated, so we will split the analysis into several cases $c=1,c=-1$, etc., and record each case in a separate theorem. The goal in each case is to make more explicit the expressions of Theorem~\ref{thm:cBU_Char} for the Gold functions.

We  need some notations below.  For $1\leq k<n$, let $\alpha,\beta\in\F_{p^n}$, $L_{\alpha,\beta}(x)=(\alpha+\beta)x^{p^{2k}}+(\beta^{p^{n-k}}+\beta)x$ and $\sY_1$, $\sY_2$ be the set of $(\alpha,\beta)\in\F_{p^n}^{*2}$, where $L_{\alpha,\beta}$, respectively,   $L_{-\alpha c,-\beta c^{-1}}$   are not permutations.

Further, let $\mathscr{A}_1$ be the set of all $(\alpha,\beta)\in\F_{p^n}^{*2}$ such that $L_{\alpha,\beta}(x)=-(\beta+\beta^{p^k})$ has a root $x_{\alpha,\beta}$,  and
$\alpha,\beta$ satisfy (with $d=\gcd(2k,n)$)
\[
(-1)^{\frac{n}{d}}\left(\frac{\beta^{p^{n-k}}+\beta}{\alpha+\beta} \right)^{\frac{p^n-1}{p^d-1}}=1,
\]
(hence, $L_{\alpha,\beta}$ is not a permutation polynomial~\cite{ZWW20}). Observe that the left hand side expression is just the relative norm from $\F_{p^n}$ to $\F_{p^d}$ of the argument.
Similarly, let $\sA_2$ be the set of all $(\alpha,\beta)\in\F_{p^n}^{*2}$ such that $L_{-\alpha c,-\beta c^{-1}}(x)=(\beta c^{-1}+(\beta c^{-1})^{p^k})$ has a root $x_{-\alpha c,-\beta c^{-1}}$,  and
$\alpha,\beta$ satisfy  
\[
(-1)^{\frac{n}{d}}\left(\frac{(c^{-1}\beta)^{p^{n-k}}+c^{-1}\beta}{\alpha c+\beta c^{-1}} \right)^{\frac{p^n-1}{p^d-1}}=1
\]
(hence, $L_{-\alpha c,-\beta c^{-1}} $ is not a permutation polynomial).

We showed in~Theorem~\ref{thm:cBU_Char} that ${_c}B_F(a,ab)$ equals (we let $q=p^n$)
\begin{equation}
\label{eq:eq32}
  \frac{1}{q}\left({_c}\Delta_{F}(1,b)+{_{c^{-1}}}\Delta_{F}(1,b)\right)+1 +\frac1{q^2}\sum_{\alpha,\beta\in\F_q,\alpha\beta\neq 0} \chi_1(-b(\alpha+\beta))\, S_{\alpha,\beta}\, S_{-\alpha c,-\beta c^{-1}},
\end{equation}
where
$
S_{\alpha,\beta}=\sum_{x\in \F_q} \chi_1\left(\alpha x^d\right)\chi_1\left(\beta(x+1)^d\right)$, and $d=p^k+1$. We let $\displaystyle T_b=\sum_{\alpha,\beta\in\F_q,\alpha\beta\neq 0} \chi_1(-b(\alpha+\beta))\, S_{\alpha,\beta}\, S_{-\alpha c,-\beta c^{-1}}$.

We now  concentrate on $S_{\alpha,\beta}$, for $\alpha\beta\neq 0$. Using the fact that $\chi_1(u^p)=\chi_1(u)$ for $u\in\F_q$, we compute
\allowdisplaybreaks
\begin{align*}
S_{\alpha,\beta}
&=\sum_{x\in \F_q} \chi_1\left(\alpha x^{p^k+1}+\beta(x+1)^{p^k+1}\right)\\
&=\sum_{x\in \F_q} \chi_1\left((\alpha+\beta)  x^{p^k+1}+ \beta x^{p^k} +\beta x+\beta)\right)\\
&=\sum_{x\in \F_q} \chi_1\left((\alpha+\beta)  x^{p^k+1}\right) \chi_1\left( (\beta^{p^{n-k}} x)^{p^k}\right)\chi_1( \beta x+\beta)\\
&=\sum_{x\in \F_q} \chi_1\left((\alpha+\beta)  x^{p^k+1}\right) \chi_1\left( \beta^{p^{n-k}} x\right)\chi_1( \beta x+\beta)\\
&=\sum_{x\in \F_q} \chi_1\left((\alpha+\beta)  x^{p^k+1}\right) \chi_1\left( (\beta^{p^{n-k}} +\beta) x+\beta\right)\\
&=\chi_1(\beta) \sum_{x\in \F_q} \chi_1\left((\alpha+\beta)  x^{p^k+1}+ (\beta^{p^{n-k}} +\beta) x\right).
\end{align*}
Let  $A=\alpha+\beta, B=\beta^{p^{n-k}}+\beta$ (recall $\alpha\beta\neq 0$). 
If $\alpha=-\beta=\beta^{p^{n-k}}$ (the last identity can only happen for  $\frac{n}{e}$ even, where $e=\gcd(n,k)$), then  $S_{\alpha,\beta}=q\chi_1(\beta)$ (there are $p^e-1$ such nonzero $\beta$'s, since $\beta\neq 0,\beta^{p^{n-k}}+\beta=0$ is equivalent to $\beta^{p^k-1}+1=0$). If $\alpha=-\beta\neq \beta^{p^{n-k}}$, then $S_{\alpha,\beta}=0$. If   $\alpha\neq -\beta= \beta^{p^{n-k}}$, then we  use~\cite[Theorem 1 and 2]{Co98_1} (observe that the case $\frac{n}{e}$ odd does not happen), obtaining that when  $\frac{n}{e}$ is even (thus, $n=2m$), then  (with $A=\alpha+\beta$; we simplify a bit the original statement)
 \[
S_{\alpha,\beta}=
\begin{cases}
 (-1)^{\frac{m}{e}}\, p^m\,\chi_1(\beta) &\text{ if } A^{\frac{q-1}{p^e+1}}\neq (-1)^{\frac{m}{e}}\\ 
 (-1)^{\frac{m}{e}+1}\, p^{m+e}\,\chi_1(\beta) &\text{ if } A^{\frac{q-1}{p^e+1}}= (-1)^{\frac{m}{e}}. 
\end{cases}
\]

We next assume that $\alpha\neq -\beta\neq \beta^{p^{n-k}}$.
We shall now be using Theorem~\ref{thm:Co98},
which gives explicitly the sum $\mathscr{S}_k(A,B)=\sum_{x\in \F_q} \chi_1\left(A x^{p^k+1}+ B x\right)$ depending upon whether $L_{\alpha,\beta}(x)=A^{p^k}x^{p^{2k}}+Ax$  is a permutation polynomial or not. 

Now, it is known that a linearized polynomial of the form $L_r(x)=x^{p^r}+\gamma x\in\F_{p^n}$ is a permutation polynomial if and only if the relative norm $N_{\F_{p^n}/\F_{p^d}}(\gamma)\neq 1$, that is, $(-1)^{n/d} \gamma^{(p^n-1)/(p^d-1)}\neq 1$, where $d=\gcd(n,r)$. For our polynomial $L_{\alpha,\beta}(x)=(\alpha+\beta)x^{p^{2k}}+(\beta^{p^{n-k}}+\beta)x$, (dividing by $\alpha+\beta\neq 0$) the previous nonpermutability condition becomes (with $d=\gcd(2k,n)$)
\begin{equation}
\label{eq:pp_eq}
(-1)^{\frac{n}{d}}\left(\frac{\beta^{p^{n-k}}+\beta}{\alpha+\beta} \right)^{\frac{p^n-1}{p^d-1}}=1.
\end{equation}
Surely, there are  $\frac{p^n-1}{p^d-1}$ roots for the equation $x^{\frac{p^n-1}{p^d-1}}=(-1)^{\frac{n}{d}}$, forming a set $\sX_1$ of cardinality $\frac{p^n-1}{p^d-1}$. We then see that for an arbitrary $\gamma\in \sX_1$, and any $\beta\in\F_{p^n}$, then there is a unique $\alpha\in\F_{p^n}$ such that $\frac{\beta^{p^{n-k}}+\beta}{\alpha+\beta}=\gamma$, namely $\alpha=\frac{\beta^{p^{n-k}}+\beta-\beta \gamma}{\gamma}$. Therefore, there are $\leq \frac{p^n(p^n-1)}{p^d-1}$ pairs $(\alpha,\beta)$ forming a set $\sY_1$ (with the restrictions $\alpha\neq -\beta\neq \beta^{p^{n-k}}$)  such that $L_{\alpha,\beta}$ is not a permutation (the reason that the number of pairs is not precisely $ \frac{p^n(p^n-1)}{p^d-1}$ is because more than one $\beta$ may generate the same $\alpha$ if the linearized $x^{p^{n-k}}+(1-\gamma)x$ is not a permutation polynomial).   Let $\mathscr{A}_1\subseteq \sY_1$ be the subset of all $(\alpha, \beta)\in\sY_1$ such that $L_{\alpha,\beta}(x)=-(\beta+\beta^{p^k})$ has at least a root $x_{\alpha,\beta}$. 
Surely, if $L_{\alpha,\beta}$ is not a permutation on $\F_q$ and $L_{\alpha,\beta}(x)=-(\beta+\beta^{p^k})$ has no root, then  $S_{\alpha,\beta}=0$.

If $L_{\alpha,\beta}$ is not a PP (abbreviation of ``permutation polynomial''), but the  linearized equation has a root (hence, by Theorem~\ref{thm:Co98},   $\frac{n}{e}$ is even), then for all  $(\alpha,\beta)\in \sA_1$ (note that the cardinality of $|\sA_1|\leq \frac{p^n(p^n-1)}{p^d-1}$), then 
\[
S_{\alpha,\beta}=(-1)^{\frac{n}{2e}+1}p^{\frac{n}{2}+e} \chi_1(\beta) \overline{\chi_1\left((\alpha+\beta) x_{\alpha,\beta}^{p^k+1}\right)}.
\]

In addition, again, by Theorem~\ref{thm:Co98}, when $\frac{n}{e}$ is even   and $L_{\alpha,\beta}$ is a permutation on $\F_q$ and $x_{\alpha,\beta}$ is the root of $L_{\alpha,\beta}(x)=-(\beta+\beta^{p^k})$ then 
\[
S_{\alpha,\beta}=(-1)^{\frac{n}{2e}}p^{\frac{n}{2}}  \chi_1(\beta) \overline{\chi_1\left((\alpha+\beta) x_{\alpha,\beta}^{p^k+1}\right)}.
\]

 Finally, if  $\frac{n}{e}$ is odd and $L_{\alpha,\beta}$ is a permutation on $\F_q$ and $x_{\alpha,\beta}$ is the root of $L_{\alpha,\beta}(x)=-(\beta+\beta^{p^k})$, then $S_{\alpha,\beta}$ equals
\begin{align*}
(-1)^{n-1} \sqrt{q}\,\chi_1(\beta)\eta(-\alpha-\beta)\,\overline{\chi_1((\alpha+\beta)x_{\alpha,\beta}^{p^k+1})},\ &\text{ if } p\equiv 1\pmod 4\\
(-1)^{n-1} \imath^{3n} \sqrt{q}\,\chi_1(\beta)\eta(-\alpha-\beta)\,\overline{\chi_1((\alpha+\beta)x_{\alpha,\beta}^{p^k+1})},\ &\text{ if } p\equiv 3\pmod 4.
\end{align*}

We  take $\sX_2,\sY_2,\sA_2$ to be the corresponding sets as above, where $L_{\alpha,\beta}$ is replaced by $L_{-\alpha c,-\beta c^{-1}}$, etc. 

\subsection{The case $c=1$}

 \begin{thm}
\label{thm:c=1}
Let $F(x)=x^{p^k+1}$, $1\leq k<n$, be the Gold function on $\F_{p^n}$, $p$ and odd prime,  $n\geq 2$, and $c=1$. The Boomerang Connectivity Table entry of $F$ at $(a,ab$) is 
$\displaystyle _c\cB_F(a,ab)=\frac{2}{q}\Delta_{F}(1,b)+1 +\frac1{q^2}T_b$, where:
\begin{itemize}
\item[$(i)$] If $\frac{n}{e}$ is odd,   then  
\begin{align*}
T_b&=(-1)^{n\frac{p+1}{2}} p^n \sum_{(\alpha,\beta)\in\bar\sY_1} \chi_1(-b(\alpha+\beta)).
\end{align*}
\item[$(ii)$] If  $\frac{n}{e}$ even, then, with $A=\alpha+\beta$ and $\Sigma_1=\displaystyle \sum_{\substack{\alpha,\beta\in\F_q^*\\ A^{\frac{q-1}{p^e+1}}= (-1)^{\frac{m}{e}}}} \chi_1(-bA)$,
\allowdisplaybreaks
\begin{align*}
T_b=&\left(p^{2n+e}-2p^{2n}+p^n \right)
 +\left(p^{2n+2e}-p^{n+2e}-p^{2n}+p^n \right) \Sigma_1\\
& +p^{n+2e} \sum_{(\alpha,\beta)\in\sA_1}   \chi_1\left(-b(\alpha+\beta)\right)
+p^n\sum_{(\alpha,\beta)\in\bar \sY_1} \chi_1\left(-b(\alpha+\beta)\right).
\end{align*}
\end{itemize}

\end{thm}
 
 \begin{proof}
Observe that the conditions $\alpha=-\beta$, $-\beta=\beta^{p^{n-k}}$, and $-c\alpha=c\beta$, $c\beta=(-c\beta)^{p^{n-k}}$ are equivalent, when $c=1$. 

First, if  $\frac{n}{e}$ is odd, and $p\equiv 1\pmod 4$,  $\alpha\neq -\beta\neq \beta^{p^{n-k}}$ and $L_{\alpha,\beta}$ is  a permutation (as well as, $L_{-\alpha,-\beta}$) (recall that $(\alpha,\beta)\in \overline{\sY_1}$), then (since $x_{\alpha,\beta}=x_{-\alpha,-\beta}$),
\allowdisplaybreaks
\begin{align*}
S_{\alpha,\beta}S_{-\alpha,-\beta}=&
 q\,\chi_1(\beta)\eta(-\alpha-\beta)\,\overline{\chi_1((\alpha+\beta)x_{\alpha,\beta}^{p^k+1})}\\
&\qquad \cdot \chi_1(-\beta)\eta(\alpha+\beta) \,\overline{\chi_1(-(\alpha+\beta)x_{\alpha,\beta}^{p^k+1})}= q (-1)^n,
\end{align*}
and so,
\[
T_{b}=(-1)^n p^{n}\sum_{(\alpha,\beta)\in\bar\sY_1} \chi_1(-b(\alpha+\beta)).
\]
Similarly, in the same case, if $p\equiv 3\pmod 4$, then  (there are two extra copies of $\imath^{3n}$ rendering a factor of $(-1)^n$)
\begin{align*}
S_{\alpha,\beta}S_{-\alpha,-\beta}&= (-1)^n  q\,\eta(-1)=q,
\end{align*}
therefore,
\[
T_{b}=p^{n}\sum_{(\alpha,\beta)\in\bar\sY_1} \chi_1(-b(\alpha+\beta)).
\]

Therefore, when $\frac{n}{e}$ is odd, we can uniquely write this as 
\begin{align*}
T_b&=(-1)^{n\frac{p+1}{2}} p^n \sum_{(\alpha,\beta)\in\bar\sY_1} \chi_1(-b(\alpha+\beta)).
\end{align*}

If  $0\neq \alpha=-\beta= \beta^{p^{n-k}}$ (there are $p^e-1$ such roots $\beta$), so, $\frac{n}{e}$ is even, then $S_{\alpha,\beta}=q\chi_1(\beta)$, $S_{-\alpha,-\beta}=q\chi_1(-\beta)$, and so,
\allowdisplaybreaks
\begin{align*}
T_{b,1}&=\sum_{0\neq \alpha=-\beta= \beta^{p^{n-k}}}\chi_1(-b(\alpha+\beta)) S_{\alpha,\beta}S_{-\alpha,-\beta}\\
& = p^{2n} \sum_{0\neq \alpha=-\beta= \beta^{p^{n-k}}} \chi_1(-b(\alpha+\beta))\chi_1(\beta)\chi_1(-\beta) =p^{2n}(p^e-1).
\end{align*}

Assume $\alpha\neq-\beta= \beta^{p^{n-k}}$, so $\frac{n}{e}$ is even. We observe that either both $\alpha+\beta,-(\alpha+\beta)$ satisfy  $X^{\frac{q-1}{p^e+1}}=(-1)^{\frac{m}{e}}$, or none will do. Furthermore,  
\begin{align*}
S_{\alpha,\beta}S_{-\alpha,-\beta}= 
\begin{cases} 
p^{n}  &\text{ if } A^{\frac{q-1}{p^e+1}}\neq (-1)^{\frac{m}{e}}\\
p^{n+2e}&\text{ if } A^{\frac{q-1}{p^e+1}}=(-1)^{\frac{m}{e}}.
\end{cases}
\end{align*}
Thus, if $A^{\frac{q-1}{p^e+1}}= (-1)^{\frac{m}{e}}$,  then 
\begin{align*}
T_{b,2}&=p^{n+2e} \sum_{\substack{\alpha,\beta\in\F_q^*\\ A^{\frac{q-1}{p^e+1}}= (-1)^{\frac{m}{e}}}} \chi_1(-b(\alpha+\beta))  \\
&=p^{n+2e}\sum_{A ,A^{\frac{q-1}{p^e+1}}=(-1)^{\frac{m}{e}}}\sum_{\beta\in\F_q^*}  \chi_1(-bA )\\
&=p^{n+2e}(p^n-1) \Sigma_1.
\end{align*}
If $A^{\frac{q-1}{p^e+1}}\neq  (-1)^{\frac{m}{e}}$, $A\neq 0$ (since $\alpha\neq -\beta$), then
\begin{align*}
T_{b,3}&=p^{n} \sum_{\substack{\alpha,\beta\in\F_q^*\\ A^{\frac{q-1}{p^e+1}}\neq (-1)^{\frac{m}{e}}}} \chi_1(-b(\alpha+\beta)) =p^n(p^n-1) \sum_{\substack{A\neq 0\\ A^{\frac{q-1}{p^e+1}}\neq (-1)^{\frac{m}{e}}}} \chi_1(-bA)\\
&=  p^n(p^n-1) \left(\sum_{A\in\F_q}\chi_1(-bA)-  \sum_{A,A^{\frac{q-1}{p^e+1}}= (-1)^{\frac{m}{e}}} \chi_1(-bA)-1\right)\\
& =- p^n(p^n-1)\left(1+ \Sigma_1\right).
\end{align*}

Now, let  $\alpha\neq-\beta\neq \beta^{p^{n-k}}$ and $\frac{n}{e}$ even. Then (we assume that $x_{\alpha,\beta}$ is a root of $L_{\alpha,\beta}(x)=-(\beta+\beta^{p^k})$, if it exists; observe also that $-x_{\alpha,\beta}$ is a root for $L_{-\alpha,-\beta}(x)=-((-\beta)+(-\beta)^{p^k})$, as well),
\begin{align*}
S_{\alpha,\beta}S_{-\alpha,-\beta}= 
\begin{cases} 
p^{n} 
&\text{ if   $L_{\alpha,\beta}$ is PP}\\
p^{n+2e}
&\text{ if   $L_{\alpha,\beta}$ is not PP}.
\end{cases}
\end{align*}
In this case, then,
\begin{align*}
T_{b,4}&=p^{n+2e} \sum_{(\alpha,\beta)\in\sA_1}   \chi_1\left(-b(\alpha+\beta)\right)
+p^n\sum_{(\alpha,\beta)\in\bar \sY_1} \chi_1\left(-b(\alpha+\beta)\right).
\end{align*}
Thus, when $\frac{n}{e}$ is even, 
\[
T_b=T_{b,1}+T_{b,2}+T_{b,3}+T_{b,4}.
\]
The theorem is shown.
\end{proof}
\begin{rem}
As an example, we took $p=n=3$, $k=2$, $b=2$, and easily found in less than a second, using SageMath on a Macbook Pro \text{\rm I}$7$ with $16$\text{\rm GB} of \text{\rm RAM}, the set of all pairs $(\alpha,\beta)$, such that $\alpha\neq -\beta\neq \beta^{p^{n-k}}$ and $L_{\alpha,\beta}$ is  a permutation, and computed the value  of $\displaystyle _c\cB_F(a,2a)=2$, via Theorem~$6 (i)$, for all $a$.
\end{rem}

\subsection{The case $c=-1$}
Below, we take  $x_{\alpha,\beta}$ to be the root of $L_{\alpha,\beta}(x)=-(\beta^{p^k}+\beta)$, which always exists if $(\alpha,\beta)\in \cA_1\cup\bar\sY_1$, and 
\[
  \Sigma_1=\sum_{A,A^{\frac{q-1}{p^e+1}}=(-1)^{\frac{m}{e}}} \chi_1(-bA),\quad   \Sigma_2= \sum_{\beta^{p^k-1}+1=0} \chi_1(2\beta).
\]
\begin{thm}
Let $F(x)=x^{p^k+1}$, $1\leq k<n$, be the Gold function on $\F_{p^n}$, $p$ and odd prime,  $n\geq 2$, and $c=-1$. In addition to $\Sigma_1$ from Theorem~\textup{\ref{thm:c=1}}, we let $\displaystyle \Sigma_2= \sum_{\beta^{p^k-1}+1=0} \chi_1(2\beta)$. The $(-1)$-Boomerang Connectivity Table entry of $F$ at $(a,ab$) is 
$\displaystyle _c\cB_F(a,ab)=\frac{2}{q } \left({_{-1}}\Delta_{F}(1,b)\right)+1 +\frac1{q^2}T_b$, where:
\begin{itemize}
\item[$(i)$] If $\frac{n}{e}$ is odd, then 
\begin{align*}
T_b&=(-1)^{n\frac{p-1}{2}} p^{n} \sum_{(\alpha,\beta)\in\bar \sY_1}  \chi_1\left(2\beta-A\left(b+2x_{\alpha,\beta}^{p^k+1}\right)\right).
\end{align*}
\item[$(ii)$]  If $\frac{n}{e}$ is even,
\allowdisplaybreaks
\begin{align*}
T_b&=p^n(p^n-1)\Sigma_1+p^n(p^{2e}-1)\Sigma_1\Sigma_2\\
&\quad +p^{n+2e} \sum_{(\alpha,\beta)\in\sA_1}    \chi_1\left(2\beta-(\alpha+\beta)\left(2x_{\alpha,\beta}^{p^k+1}+b\right)\right)\\
&\quad   +p^n\sum_{(\alpha,\beta)\in\bar \sY_1}   \chi_1\left(2\beta-(\alpha+\beta)\left(2x_{\alpha,\beta}^{p^k+1}+b\right)\right).
\end{align*}
\end{itemize}
\end{thm}
\begin{proof}
 In this case, since $c=-1$, $S_{\alpha,\beta}=S_{-\alpha c,-\beta c^{-1}}$. 
 If $\frac{n}{e}$ is odd and $\alpha\neq -\beta\neq \beta^{p^{n-k}}$, then 
\begin{align*}
\left(S_{\alpha,\beta}\right)^2&=p^n \chi_1\left(2\beta-2(\alpha+\beta)x_{\alpha,\beta}^{p^k+1}\right) \text{ if  $L_{\alpha,\beta}$ is   PP,  $p\equiv 1\pmod 4$},\\
\left(S_{\alpha,\beta}\right)^2&=(-1)^n p^n \chi_1\left(2\beta-2(\alpha+\beta)x_{\alpha,\beta}^{p^k+1}\right) \text{ if  $L_{\alpha,\beta}$ is  PP, $p\equiv 3\pmod 4$}.
\end{align*}
We then get, if $L_{\alpha,\beta}$ is a PP, then
\begin{align*}
T_{b}&=p^{n} \sum_{(\alpha,\beta)\in\bar \sY_1}  \chi_1\left(2\beta-A\left(b+2x_{\alpha,\beta}^{p^k+1}\right)\right)\text{ if } p\equiv 1\pmod 4,\\
T_{b}&=(-1)^n p^{n} \sum_{(\alpha,\beta)\in\bar \sY_1}  \chi_1\left(2\beta-A\left(b+2x_{\alpha,\beta}^{p^k+1}\right)\right) \text{ if } p\equiv 3\pmod 4.
\end{align*}
Therefore, if $\frac{n}{e}$ is odd, then, we can write this uniquely as
\[
T_b=(-1)^{n\frac{p-1}{2}} p^{n} \sum_{(\alpha,\beta)\in\bar \sY_1}  \chi_1\left(2\beta-A\left(b+2x_{\alpha,\beta}^{p^k+1}\right)\right).
\]

When $\frac{n}{e}$ is even, we will write $T_b=T_{b,1}+T_{b,2}+T_{b,3}+T_{b,4}$, with the partial sum's expressions  defined below.
If   $\frac{n}{e}$ is even and $0\neq \alpha=-\beta= \beta^{p^{n-k}}$, then
\begin{align*}
T_{b,1}&= \sum_{0\neq \alpha=-\beta= \beta^{p^{n-k}}}\chi_1(-b(\alpha+\beta)) \left(S_{\alpha,\beta}\right)^2  \\ 
&=p^{2n} \sum_{\beta,\beta^{p^k-1}+1=0} \chi_1(2\beta)=p^{2n}  \Sigma_2.
\end{align*}
If $\frac{n}{e}$ is even ($n=2m$) and $0\neq \alpha\neq -\beta= \beta^{p^{n-k}}$, then
\begin{align*}
\left(S_{\alpha,\beta}\right)^2  = 
\begin{cases} 
p^{n+2e}\chi_1(2\beta) &\text{ if } A^{\frac{q-1}{p^e+1}}=(-1)^{\frac{m}{e}}\\
p^{n}\chi_1(2\beta)  &\text{ if } A^{\frac{q-1}{p^e+1}}\neq (-1)^{\frac{m}{e}}.
\end{cases}
\end{align*}
Thus, if $\frac{n}{e}$ is even, $0\neq \alpha\neq -\beta= \beta^{p^{n-k}}$ and $A^{\frac{q-1}{p^e+1}}= (-1)^{\frac{m}{e}}$,  then 
\begin{align*}
T_{b,2}&=p^{n+2e} \sum_{\substack{\alpha,\beta\in\F_q^*,\beta^{p^k-1}+1=0 \\ A^{\frac{q-1}{p^e+1}}= (-1)^{\frac{m}{e}}}} \chi_1(-b(\alpha+\beta)) \chi_1(2\beta)\\
&=p^{n+2e} \sum_{A,A^{\frac{q-1}{p^e+1}}=(-1)^{\frac{m}{e}}} \chi_1(-bA)  \sum_{\beta,\beta^{p^k-1}+1=0} \chi_1(2\beta)=p^{n+2e}\Sigma_1\Sigma_2.
\end{align*}
If $ \frac{n}{e}$ is even, $0\neq \alpha\neq -\beta= \beta^{p^{n-k}}$ and $A^{\frac{q-1}{p^e+1}}\neq  (-1)^{\frac{m}{e}}$  (note that $A\neq 0$), then
\begin{align*}
T_{b,3}&=p^{n} \sum_{\substack{\alpha,\beta\in\F_q^*,\beta^{p^k-1}+1=0 \\ A\neq 0,A^{\frac{q-1}{p^e+1}}\neq (-1)^{\frac{m}{e}}}}  \chi_1(2\beta-b(\alpha+\beta)) \\
&=p^n \sum_{\beta,\beta^{p^k-1}+1=0} \chi_1(2\beta) \sum_{\substack{A\neq 0\\ A^{\frac{q-1}{p^e+1}}\neq (-1)^{\frac{m}{e}}}} \chi_1(-bA)\\
&=  -p^n\Sigma_2  \left(\sum_{A,A^{\frac{q-1}{p^e+1}}=(-1)^{\frac{m}{e}}} \chi_1(-bA)+1 \right)\\
&=-p^n\Sigma_2\left(\Sigma_1+1 \right).
\end{align*}
Now, we let   $\frac{n}{e}$ be even and $\alpha\neq-\beta\neq \beta^{p^{n-k}}$. Then (we assume that $x_{\alpha,\beta}$ is a root of $L_{\alpha,\beta}(x)=-(\beta+\beta^{p^k})$, if it exists),
\begin{align*}
\left(S_{\alpha,\beta}\right)^2  = 
\begin{cases} 
p^{n+2e} \chi_1(2\beta) \overline{ \chi_1\left(2Ax_{\alpha,\beta}^{p^k+1}\right)} &\text{ if   $L_{\alpha,\beta}$ is not PP}\\
p^{n}\chi_1(2\beta) \overline{\chi_1\left(2Ax_{\alpha,\beta}^{p^k+1}\right)} &\text{ if   $L_{\alpha,\beta}$ is PP}.
\end{cases}
\end{align*}
In this case, then,
\begin{align*}
T_{b,4}&=p^{n+2e} \sum_{(\alpha,\beta)\in\sA_1}   \chi_1\left(2\beta-(\alpha+\beta)\left(2x_{\alpha,\beta}^{p^k+1}+b\right)\right)\\
&\qquad\qquad\qquad  +p^n\sum_{(\alpha,\beta)\in\bar \sY_1}    \chi_1\left(2\beta-(\alpha+\beta)\left(2x_{\alpha,\beta}^{p^k+1}+b\right)\right).
\end{align*}

Therefore, if $\frac{n}{e}$ is even, then
\allowdisplaybreaks
\begin{align*}
T_b&=T_{b,1}+T_{b,2}+T_{b,3}+T_{b,4}\\
&=p^{2n}\Sigma_1+p^{n+2e}\Sigma_1\Sigma_2-p^n\Sigma_2(\Sigma_1+1)\\
&\quad +p^{n+2e} \sum_{(\alpha,\beta)\in\sA_1}    \chi_1\left(2\beta-(\alpha+\beta)\left(2x_{\alpha,\beta}^{p^k+1}+b\right)\right)\\
&\quad   +p^n\sum_{(\alpha,\beta)\in\bar \sY_1}    \chi_1\left(2\beta-(\alpha+\beta)\left(2x_{\alpha,\beta}^{p^k+1}+b\right)\right)\\
&=p^n(p^n-1)\Sigma_2+p^n(p^{2e}-1)\Sigma_1\Sigma_2\\
&\quad +p^{n+2e} \sum_{(\alpha,\beta)\in\sA_1}    \chi_1\left(2\beta-(\alpha+\beta)\left(2x_{\alpha,\beta}^{p^k+1}+b\right)\right)\\
&\quad   +p^n\sum_{(\alpha,\beta)\in\bar \sY_1}   \chi_1\left(2\beta-(\alpha+\beta)\left(2x_{\alpha,\beta}^{p^k+1}+b\right)\right).
\end{align*}
Our theorem is shown.
\end{proof}

\subsection{The case $c^{p^k-1}=1$, $c$ not equal to $\pm 1$}

Since they were treated earlier, we assume here that $c \ne\pm 1$.
In this case, the conditions $-\beta=\beta^{p^{n-k}}$, and  $-c^{-1}\beta=(-c^{-1} \beta)^{p^{n-k}}$ are equivalent. From here on until the end of the subsection, we let $A=\alpha+\beta, A'=-c\alpha-c^{-1}\beta$. We will be using below that if  $-\beta=\beta^{p^{n-k}}$ (which is the same as $\beta^{p^k-1}=-1$, or, even further, $\beta^{p^e-1}=-1$); this can happen only if $\frac{n}{e}$ is even), then $\beta^{\frac{q-1}{p^e+1}}=(-1)^{\frac{m}{e}}$. This follows from the following computation (let $n=2m=2et$, so $t=\frac{m}{e}$):
\begin{align*}
\beta^{\frac{q-1}{p^e+1}}&=\beta^{(p^e-1)\left(p^{2e(t-1)}+\cdots+p^{2e}+1\right)}=(-1)^{t}=(-1)^{\frac{m}{e}}.
\end{align*}

 We define $\displaystyle \Sigma_3=\sum_{0\neq -\beta=\beta^{p^{n-k}}}  \chi_1\left(\beta(1-c^{-1})\right)$ and use the   notation  
\allowdisplaybreaks
\[
\Sigma(L)=\sum_{(\alpha,\beta)\in L} \chi_1\left(\beta-(\alpha+\beta) x_{\alpha,\beta}^{p^k+1}\right)   
\chi_1\left(-\beta c^{-1}+(\alpha c+\beta c^{-1}) x_{-\alpha c,-\beta c^{-1}}^{p^k+1}\right).
\]
Further, $\cC,\cC'$ are the sets of $(\alpha,\beta)$ satisfying the conditions
$A^{\frac{q-1}{p^e+1}}=(-1)^{\frac{m}{e}}$, respectively, ${A'}^{\frac{q-1}{p^e+1}}=(-1)^{\frac{m}{e}}$.

\begin{thm}
Let $F(x)=x^{p^k+1}$, $1\leq k<n$, be the Gold function on $\F_{p^n}$, $p$ and odd prime,  $n\geq 2$, and $c^{p^k-1}=1$, $c\neq \pm 1$. The $c$-Boomerang Connectivity Table entry of $F$ at $(a,ab$) is 
$\displaystyle _c\cB_F(a,ab)=\frac{1}{q}\left({_c}\Delta_{F}(1,b)+{_{c^{-1}}}\Delta_{F}(1,b)\right)+1 +\frac1{q^2}T_b$, where:
\begin{itemize}
\item[$(i)$] If $\frac{n}{e}$ is odd, then 
{\small
\begin{align*}
T_b&=p^n (-1)^{\frac{n(p-1)}{2}} \sum_{(\alpha,\beta)\in\bar \sY_1\cap\bar\sY_2} \chi_1(-b A+\beta(1-c^{-1})) \\
&\qquad\qquad\qquad\qquad\qquad \cdot \eta(AA')   \overline{\chi_1\left(Ax_{\alpha,\beta}^{p^k+1} +A'x_{-\alpha c,-\beta c^{-1}}^{p^k+1}\right)}.
\end{align*}
}
\item[$(ii)$]  If $\frac{n}{e}$ is even $(n=2m)$,
\allowdisplaybreaks
{\small
\begin{align*}
&T_b=p^{n+m+e}(-1)^{\frac{m}{e}+1}  \Sigma_3+ p^{n+2e} \sum_{\substack{\alpha,\beta\in\cC\cap \cC'\\ \alpha\neq -\beta=\beta^{p^{n-k}}\\ 0\neq \alpha \neq -\beta c^{-2}\neq 0 }}   \chi_1(-b(\alpha+\beta)) \chi_1(\beta(1-c^{-1}))\\
&-p^{n+e}\sum_{\substack{\alpha,\beta\in\cC\triangle \cC'\\ \alpha\neq -\beta=\beta^{p^{n-k}}\\ 0\neq \alpha \neq -\beta c^{-2}\neq 0 }}  \chi_1(-b(\alpha+\beta)) \chi_1(\beta(1-c^{-1}))\\
& +p^n\sum_{\substack{\alpha,\beta\in\bar \cC\cap \bar\cC'\\ \alpha\neq -\beta=\beta^{p^{n-k}}\\ 0\neq \alpha \neq -\beta c^{-2}\neq 0 }}   \chi_1(-b(\alpha+\beta)) \chi_1(\beta(1-c^{-1}))
\\
&  +p^{n+m+e}(-1)^{\frac{m}{e}+1}  \sum_{\substack{\beta^{p^k-1}=-1\\ 
\alpha=-\beta c^{-2} }} \chi_1\left(\beta(1-c^{-1})(1-b(1+c^{-1})\right)\\
&+p^{n+2e}  \Sigma\left(\sA_1'\cap \sA_2'\right)-p^{n+e}\Sigma\left((\sA_1'\cap \tilde\sY_2)\cup (\sA_2'\cap\tilde \sY_1)\right)+p^n\Sigma\left(\tilde\sY_1\cap \tilde \sA_2\right).
\end{align*}
}
\end{itemize}
\end{thm}
\begin{proof}
We first assume that  $\frac{n}{e}$ is odd, $n=2m$, $e=\gcd(n,k)$, and $\alpha\neq -\beta\neq \beta^{p^{n-k}}$.  As before,
the only relevant case is   $-\alpha c\neq  \beta c^{-1}\neq (-\beta c^{-1})^{p^{n-k}}$ and both $L_{\alpha,\beta},L_{-\alpha c,-\beta c^{-1}}$ are permutations, surely on $\bar \sY_1\cap\bar\sY_2$. Then, $S_{\alpha,\beta}S_{-\alpha c,-\beta c^{-1}}$ equals
{\small
\begin{align*}
&p^n \eta(AA') \chi_1(\beta(1-c^{-1})) \overline{\chi_1\left(Ax_{\alpha,\beta}^{p^k+1} +A'x_{-\alpha c,-\beta c^{-1}}^{p^k+1}\right)},\text{ if } p\equiv 1\pmod 4\\
&p^n (-1)^n \eta(AA') \chi_1(\beta(1-c^{-1})) \overline{\chi_1\left(Ax_{\alpha,\beta}^{p^k+1} +A'x_{-\alpha c,-\beta c^{-1}}^{p^k+1}\right)},\text{ if } p\equiv 3\pmod 4,
\end{align*}
}
and so, $T_b$ equals
{\small
\[
p^n (-1)^{\frac{n(p-1)}{2}} \sum_{(\alpha,\beta)\in\bar \sY_1\cap\bar\sY_2} \chi_1(-b A+\beta(1-c^{-1}))
\eta(AA') \overline{\chi_1\left(Ax_{\alpha,\beta}^{p^k+1} +A'x_{-\alpha c,-\beta c^{-1}}^{p^k+1}\right)}.
\]
}

We continue now with $\frac{n}{e}$ being even.
If $\alpha=-\beta=\beta^{p^{n-k}}$, then surely $-\alpha c\neq \beta c^{-1}=(-\beta c^{-1} )^{p^{n-k}}$.   Observe that if $\alpha=-\beta=\beta^{p^{n-k}}$, then ${A'}^{\frac{q-1}{p^e+1}}=(-1)^{\frac{m}{e}}(c-c^{-1})^{\frac{q-1}{p^e+1}} =(-1)^{\frac{m}{e}}(c^2-1)^{\frac{q-1}{p^e+1}}  $. 
Using the fact that $c^{p^e}=c$, then (with $m=te$)
\begin{align*}
(c^2-1)^{\frac{q-1}{p^e+1}}&=(c^2-1)^{(p^e-1)(p^{2e(t-1)}+\cdots+p^{2e}+1)}\\
&= \left((c^2-1)^{p^{2e(t-1)}+\cdots+p^{2e}+1}\right)^{p^e-1}= \left(\prod_{i=1}^t (c^2-1)^{p^{2e(t-i)}}\right)^{p^e-1}\\
&= \left(\prod_{i=1}^t (c^{2p^{2e(t-i)}}-1)\right)^{p^e-1}= \left(\prod_{i=1}^t (c^{2}-1)\right)^{p^e-1}\\
&=(c^2-1)^{t(p^e-1)}=\left(\frac{(c^2-1)^{p^e}}{c^2-1} \right)^t
=\left(\frac{\left(c^{p^e}\right)^2-1}{c^2-1} \right)^t=1,
\end{align*}
so ${A'}^{\frac{q-1}{p^e+1}}=(-1)^{\frac{m}{e}}$ holds automatically.

Recall that  $\displaystyle \Sigma_3=\sum_{0\neq -\beta=\beta^{p^{n-k}}}  \chi_1\left(\beta(1-c^{-1})\right)$. When  $\frac{n}{e}$ is even, we will write $T_b=T_{b,1}+T_{b,2}+T_{b,3}$, with the partial sum's expressions  defined below.
Thus, if $\frac{n}{e}$ is even and $\alpha=-\beta=\beta^{p^{n-k}}$ (thus, $-\alpha c\neq \beta c^{-1}=(-\beta c^{-1} )^{p^{n-k}}$), then
\[
S_{\alpha,\beta}S_{-\alpha c,-\beta c^{-1}} 
=
p^{n+m+e}(-1)^{\frac{m}{e}+1} \chi_1\left(\beta(1-c^{-1})\right),
\] 
and so,  
\allowdisplaybreaks
\begin{align*}
T_{b,1}&=p^{n+m+e}(-1)^{\frac{m}{e}+1} \sum_{\substack{0\neq \alpha=-\beta\\
-\beta=\beta^{p^{n-k}}}} \chi_1(-b(\alpha+\beta))  \chi_1\left(\beta(1-c^{-1})\right)  \\
&= p^{n+m+1}(-1)^{\frac{m}{e}+1} \sum_{0\neq -\beta=\beta^{p^{n-k}}}  \chi_1\left(\beta(1-c^{-1})\right) =p^{n+m+e}(-1)^{\frac{m}{e}+1}  \Sigma_3.
\end{align*}

If $\frac{n}{e}$ is even and $\alpha\neq -\beta=\beta^{p^{n-k}}$, there are two subcases: $-\alpha c\neq \beta c^{-1}=(-\beta c^{-1})^{p^{n-k}}$, and $-\alpha c= \beta c^{-1}=(-\beta c^{-1})^{p^{n-k}}$ (this can only happen for $\alpha=-\beta c^{-2}$). Here, for easy writing, we let $\cC,\cC'$ be the set of $(\alpha,\beta)$ satisfying the conditions
$A^{\frac{q-1}{p^e+1}}=(-1)^{\frac{m}{e}}$, respectively, ${A'}^{\frac{q-1}{p^e+1}}=(-1)^{\frac{m}{e}}$.
We write $\cC\triangle\cC'= (\cC\setminus\cC')\cup (\cC'\setminus\cC)$ for the symmetric difference.

In the first subcase ($\alpha\neq -\beta c^{-2}$), we get
\[
S_{\alpha,\beta}S_{-\alpha c,-\beta c^{-1}} = 
\begin{cases}
p^{n+2e}  \chi_1(\beta(1-c^{-1})) &\text{ on } \cC\cap \cC'\\
-p^{n+e}  \chi_1(\beta(1-c^{-1})) &\text{ on } \cC\triangle\cC'\\
p^{n}  \chi_1(\beta(1-c^{-1})) &\text{ on } \bar\cC\cap\bar\cC'.
\end{cases}
\]
We now look at the second subcase ($\alpha=-\beta c^{-2}$). Note that in this case,   $A^{\frac{q-1}{p^e+1}}=(-1)^{\frac{m}{e}}$ reduces to $(c^2-1)^{\frac{q-1}{p^e+1}}=1$, and from our previous computation, $A^{\frac{q-1}{p^e+1}}=(-1)^{\frac{m}{e}}$ holds automatically.

Thus,
\allowdisplaybreaks
\begin{align*}
S_{\alpha,\beta}S_{-\alpha c,-\beta c^{-1}} &=p^{n}\chi_1\left(-\beta c^{-1}\right)
(-1)^{\frac{m}{e}+1} p^{m+e} \chi_1(\beta) \\
&= 
(-1)^{\frac{m}{e}+1} p^{n+m+e}\chi_1(\beta(1-c^{-1})),
\end{align*}
and so,
\allowdisplaybreaks
\begin{align*}
T_{b,2}&=p^{n+2e} \sum_{\substack{\alpha,\beta\in\cC\cap \cC'\\ \alpha\neq -\beta=\beta^{p^{n-k}}\\ 0\neq \alpha \neq -\beta c^{-2}\neq 0 }}   \chi_1(-b(\alpha+\beta)) \chi_1(\beta(1-c^{-1}))\\
&\quad -p^{n+e}\sum_{\substack{\alpha,\beta\in\cC\triangle \cC'\\ \alpha\neq -\beta=\beta^{p^{n-k}}\\ 0\neq \alpha \neq -\beta c^{-2}\neq 0 }}  \chi_1(-b(\alpha+\beta)) \chi_1(\beta(1-c^{-1}))\\
&\quad+p^n\sum_{\substack{\alpha,\beta\in\bar \cC\cap \bar\cC'\\ \alpha\neq -\beta=\beta^{p^{n-k}}\\ 0\neq \alpha \neq -\beta c^{-2}\neq 0 }}   \chi_1(-b(\alpha+\beta)) \chi_1(\beta(1-c^{-1}))
\\
&\quad +p^{n+m+e}(-1)^{\frac{m}{e}+1}  \sum_{\substack{\beta^{p^k-1}=-1\\ 
\alpha=-\beta c^{-2} }} \chi_1(\beta(1-c^{-1}))\chi_1(-b(\beta-\beta c^{-2})).
\end{align*}

Next, when $\alpha\neq -\beta\neq \beta^{p^{n-k}}$, the case $-\alpha c=\beta c^{-1}\neq (-\beta c^{-1})^{p^{n-k}}$ renders $S_{-\alpha c,-\beta c^{-1}} =0$. Thus, it is sufficient to assume next, when $\frac{n}{e}$ is even, that $\alpha\neq -\beta\neq \beta^{p^{n-k}}$ and $-\alpha c\neq \beta c^{-1}\neq (-\beta c^{-1})^{p^{n-k}}$.
We first investigate the condition from Equation~\eqref{eq:pp_eq} when $L_{\alpha,\beta}, L_{-\alpha c,-\beta c^{-1}}$ are PP, under $c^{p^k-1}=1$, that is, $c^{p^e-1}=1$.
We compute ($n=dt=2et$, since $\frac{n}{e}$ is even), using $\frac{p^n-1}{p^d-1}=p^{2e(t-1)}+\cdots+p^{2e}+1$,
\begin{align*}
&\left(\frac{(\beta c^{-1})^{p^{n-k}} +\beta c^{-1}}{\alpha c+\beta c^{-1}}  \right)^{\frac{p^n-1}{p^d-1}}
=\left(\frac{ c^{-p^{n-k}} \beta^{p^{n-k}} +\beta c^{-1}}{\alpha c+\beta c^{-1}}  \right)^{\frac{p^n-1}{p^d-1}}\\
&=\left(\frac{ c^{1-p^{n-k}} \beta^{p^{n-k}} +\beta }{\alpha c^2+\beta }  \right)^{\frac{p^n-1}{p^d-1}}
= \left(\frac{ \beta^{p^{n-k}} +\beta }{\alpha c^2+\beta }  \right)^{\frac{p^n-1}{p^d-1}},\text{ since } p^e-1\,|\,p^{n-k}-1.
\end{align*}
Summarizing, $L_{\alpha,\beta},L_{-\alpha c,-\beta c^{-1}}$ are not PP if and only if  $\left(\frac{ \beta^{p^{n-k}} +\beta }{\alpha +\beta }  \right)^{\frac{p^n-1}{p^d-1}}=(-1)^{\frac{n}{d}}$, respectively,  $\left(\frac{ \beta^{p^{n-k}} +\beta }{\alpha c^2+\beta }  \right)^{\frac{p^n-1}{p^d-1}}=(-1)^{\frac{n}{d}}$. 

 We are now ready to find the relevant products. Given the prior definition of the sets $\sA_i,\bar \sY_i$, $i=1,2$, we modify them to impose also $\alpha\neq \beta c^{-2}$, and write them as $ \sA_i', \tilde \sY_i$, $i=1,2$.
 First, 
\[
S_{\alpha,\beta}S_{-\alpha c,-\beta c^{-1}}=p^{n+2e} \chi_1\left(\beta-(\alpha+\beta) x_{\alpha,\beta}^{p^k+1}\right) \chi_1\left(-\beta c^{-1}+(\alpha c+\beta c^{-1}) x_{-\alpha c,-\beta c^{-1}}^{p^k+1}\right),
\]
if neither $L_{\alpha,\beta}, L_{-\alpha c,-\beta c^{-1}}$ is  PP (thus, $(\alpha,\beta)\in\sA_1'\cap \sA_2')$. Secondly,
\[
S_{\alpha,\beta}S_{-\alpha c,-\beta c^{-1}}=-p^{n+e} \chi_1\left(\beta-(\alpha+\beta) x_{\alpha,\beta}^{p^k+1}\right) \chi_1\left(-\beta c^{-1}+(\alpha c+\beta c^{-1}) x_{-\alpha c,-\beta c^{-1}}^{p^k+1}\right),
\]
if exactly one of $L_{\alpha,\beta}, L_{-\alpha c,-\beta c^{-1}}$ is   PP (thus, $(\alpha,\beta)\in(\sA_1'\cap \tilde\sY_2)\cup (\sA_2'\cap\tilde \sY_1)$). Lastly, 
\[
S_{\alpha,\beta}S_{-\alpha c,-\beta c^{-1}}=p^{n} \chi_1\left(\beta-(\alpha+\beta) x_{\alpha,\beta}^{p^k+1}\right) \chi_1\left(-\beta c^{-1}+(\alpha c+\beta c^{-1}) x_{-\alpha c,-\beta c^{-1}}^{p^k+1}\right),
\]
if both $L_{\alpha,\beta}, L_{-\alpha c,-\beta c^{-1}}$ are   PP (thus, $(\alpha,\beta)\in\tilde\sY_1\cap \tilde \sA_2$).

With the notation  
\allowdisplaybreaks
\begin{align*}
\Sigma(L)&=\sum_{(\alpha,\beta)\in L} \chi_1\left(\beta-(\alpha+\beta) x_{\alpha,\beta}^{p^k+1}\right)  
\chi_1\left(-\beta c^{-1}+(\alpha c+\beta c^{-1}) x_{-\alpha c,-\beta c^{-1}}^{p^k+1}\right),
\end{align*}
we obtain
\allowdisplaybreaks
\begin{align*}
T_{b,3}&=p^{n+2e}  \Sigma\left(\sA_1'\cap \sA_2'\right)-p^{n+e}\Sigma\left((\sA_1'\cap \tilde\sY_2)\cup (\sA_2'\cap\tilde \sY_1)\right)+p^n\Sigma\left(\tilde\sY_1\cap \tilde \sA_2\right).
\end{align*}
Therefore, when $\frac{n}{e}$ is even, then 
\[
T_b=T_{b,1}+T_{b,2}+T_{b,3},
\]
and the theorem is shown.
\end{proof}

\subsection{The general case} 
 
  We can surely find an expression for the $c$-BCT for $c^{p^k-1}\neq 1$, but it is going to be slightly complicated to write, although, as we mentioned, computing the boomerang uniformity is a difficult endeavor.
  
  As in the previous results, for $c\in\F_{p^n}$, $c^{p^k-1}\neq 0$, the $c$-Boomerang Connectivity Table entry of $F(x)=x^{p^k+1}$ at $(a,ab$) is 
\[
\displaystyle _c\cB_F(a,ab)= \frac{1}{q}\left({_c}\Delta_{F}(1,b)+{_{c^{-1}}}\Delta_{F}(1,b)\right)+1 +\frac1{q^2}T_b,
\]
 where 
$\displaystyle T_b=\sum_{\alpha,\beta\in\F_q,\alpha\beta\neq 0} \chi_1(-b(\alpha+\beta))\, S_{\alpha,\beta}\, S_{-\alpha c,-\beta c^{-1}}$
  
When $\frac{n}{e}$ is odd, then $T_b$ is the same as for the case of $c^{p^k-1}=1$, $c\neq \pm 1$, namely (recall that $A=\alpha+\beta, A'=-c\alpha-c^{-1}\beta$), $T_b$ equals
{\small
\[
p^n (-1)^{\frac{n(p-1)}{2}} \sum_{(\alpha,\beta)\in\bar \sY_1\cap\bar\sY_2} \chi_1(-b A+\beta(1-c^{-1}))
\eta(AA') \overline{\chi_1\left(Ax_{\alpha,\beta}^{p^k+1} +A'x_{-\alpha c,-\beta c^{-1}}^{p^k+1}\right)}.
\]
}

 For even $\frac{n}{e}$, we will not write the $T_b$ expressions, rather we will find just the products $S_{\alpha,\beta} S_{-\alpha c,-\beta c^{-1}}$, which obviously determine the $T_b$ expressions. 
When $\frac{n}{e}$ is even, and $\alpha=-\beta=\beta^{p^{n-k}}$, then $S_{\alpha,\beta}=p^n\chi_1(\beta)$, and either $-\alpha c=\beta c^{-1}=(-\beta c^{-1})^{p^{n-k}}$, in which case 
\[
S_{\alpha,\beta}S_{-\alpha c,-\beta c^{-1}}=p^{2n}\chi_1\left(\beta(1- c^{-1})\right),
\]
 or, $-\alpha c\neq \beta c^{-1}\neq (-\beta c^{-1})^{p^{n-k}}$, in which case, via Theorem~\ref{thm:Co98}, $S_{\alpha,\beta}S_{-\alpha c,-\beta c^{-1}}$ equals
{\small
\[
\begin{cases}
p^{n+m+e} (-1)^{\frac{m}{e}+1}\chi_1\left(\beta(1- c^{-1})+(\alpha c+\beta c^{-1})x_{-\alpha c, -\beta c^{-1}}^{p^{k+1}}\right), & (\alpha,\beta)\in\sA_2\\
p^{n+m}(-1)^{\frac{m}{e}} \chi_1\left(\beta(1- c^{-1})+(\alpha c+\beta c^{-1})x_{-\alpha c, -\beta c^{-1}}^{p^{k+1}}\right), & (\alpha,\beta)\in\bar \sY_2.
\end{cases}
\]
}
Similarly, when $\frac{n}{e}$ is even, and $-\alpha c= \beta c^{-1}= (-\beta c^{-1})^{p^{n-k}}$,  $\alpha\neq -\beta\neq \beta^{p^{n-k}}$, and   $S_{\alpha,\beta} S_{-\alpha c,-\beta c^{-1}}$ equals
{\small
\[
\begin{cases}
p^{n+m+e} (-1)^{\frac{m}{e}+1} \chi_1\left(\beta(1- c^{-1})-(\alpha +\beta )x_{\alpha , \beta}^{p^{k+1}}\right), & (\alpha,\beta)\in\sA_1\\
p^{n+m}(-1)^{\frac{m}{e}} \chi_1\left(\beta(1- c^{-1})-(\alpha +\beta )x_{\alpha , \beta}^{p^{k+1}}\right), & (\alpha,\beta)\in\bar \sY_1.
\end{cases}
\]
}
If $\frac{n}{e}$ is even,  $\alpha\neq -\beta=\beta^{p^{n-k}}$ and  $-\alpha c\neq \beta c^{-1}\neq (-\beta c^{-1})^{p^{n-k}}$ (we let here and below $\cC_1,\cC_2$ be the sets of $(\alpha,\beta)$ such that $A^{\frac{q-1}{p^e+1}}=(-1)^{\frac{m}{e}}$, respectively, ${A'}^{\frac{q-1}{p^e+1}}=(-1)^{\frac{m}{e}}$), then $S_{\alpha,\beta} S_{-\alpha c,-\beta c^{-1}}$  equals
{\small
\[
\begin{cases}
-p^{n+e}   \chi_1\left(\beta(1- c^{-1})+(\alpha c+\beta c^{-1})x_{-\alpha c, -\beta c^{-1}}^{p^{k+1}}\right), & (\alpha,\beta)\in(\bar\cC_1\cap \sA_2)\cup(\cC_1\cap\bar \sY_2)\\
p^{n}   \chi_1\left(\beta(1- c^{-1})+(\alpha c+\beta c^{-1})x_{-\alpha c, -\beta c^{-1}}^{p^{k+1}}\right), & (\alpha,\beta)\in\bar\cC_1\cap \bar\sY_2\\
p^{n+2e}   \chi_1\left(\beta(1- c^{-1})+(\alpha c+\beta c^{-1})x_{-\alpha c, -\beta c^{-1}}^{p^{k+1}}\right), & (\alpha,\beta)\in\cC_1\cap \sA_2.
\end{cases} 
\]
}
Similarly, when $\frac{n}{e}$ is even, $\alpha\neq -\beta\neq \beta^{p^{n-k}}$ and  $-\alpha c\neq \beta c^{-1}= (-\beta c^{-1})^{p^{n-k}}$, then $S_{\alpha,\beta} S_{-\alpha c,-\beta c^{-1}}$  equals
{\small
\[
\begin{cases}
-p^{n+e}   \chi_1\left(\beta(1- c^{-1})-(\alpha +\beta)x_{\alpha, \beta}^{p^{k+1}}\right), & (\alpha,\beta)\in(\bar\cC_2\cap \sA_1)\cup(\cC_2\cap\bar \sY_1)\\
p^{n}   \chi_1\left(\beta(1- c^{-1})-(\alpha +\beta)x_{\alpha, \beta}^{p^{k+1}}\right), & (\alpha,\beta)\in\bar\cC_2\cap \bar\sY_1\\
p^{n+2e}  \chi_1\left(\beta(1- c^{-1})-(\alpha +\beta)x_{\alpha, \beta}^{p^{k+1}}\right), & (\alpha,\beta)\in\cC_2\cap \sA_1.
\end{cases} 
\]
}
Finally, if $\frac{n}{e}$ is even, $\alpha\neq -\beta\neq \beta^{p^{n-k}}$ and  $-\alpha c\neq \beta c^{-1}\neq  (-\beta c^{-1})^{p^{n-k}}$, then $S_{\alpha,\beta} S_{-\alpha c,-\beta c^{-1}}$  equals
 {\footnotesize
\[
\begin{cases}
-p^{n+e}   \chi_1\left(\beta(1- c^{-1})-Ax_{\alpha, \beta}^{p^{k+1}}+A' x_{-\alpha c, -\beta c^{-1}}^{p^{k+1}}\right), & (\alpha,\beta)\in (\sA_1\cap\bar \sY_2)\cup(\bar \sY_1\cap \sA_2)\\
p^{n+2e}  \chi_1\left(\beta(1- c^{-1})-Ax_{\alpha, \beta}^{p^{k+1}}+A' x_{-\alpha c, -\beta c^{-1}}^{p^{k+1}}\right), & (\alpha,\beta)\in\sA_1\cap \sA_2\\
p^{n} \chi_1\left(\beta(1- c^{-1})-Ax_{\alpha, \beta}^{p^{k+1}}+A' x_{-\alpha c, -\beta c^{-1}}^{p^{k+1}}\right), & (\alpha,\beta)\in\bar\sY_1\cap\bar \sY_2.
\end{cases} 
\]
} 
 
 \section{Concluding Remarks}
 \label{sec5}

It would be interesting to see what the entries of the $c$-BCT are for other functions of interest, like the known PcN  or APcN (for all $c\neq 0$). We hope to see other applications and refinements of our methods, as well as continued progress in computing the $c$-differential and $c$-boomerang uniformity for other classes of functions.

\end{document}